\documentclass[12pt]{article}
\usepackage{color}
\usepackage{tocstyle}
\usepackage{graphicx,makeidx}
\usepackage{epsf}
\usepackage{amsmath,amsfonts,amssymb,latexsym}
\usepackage{enumitem}
\usepackage[width=160mm,height=230mm,left=35mm,foot=10mm]{geometry}
\usepackage{setspace}
\usepackage{mathtools}
\usepackage[mathlines]{lineno}

\def\fakeend{\end{document}}

\usepackage{chngcntr}
\counterwithin{figure}{section}
\makeindex

\newcommand{\ignore}[1]{}
\newcommand{\startClaims}{\setcounter{claim}{0}}

\newtheorem{theorem}{Theorem}[section]
\newtheorem{corollary}[theorem]{Corollary}
\newtheorem{lemma}[theorem]{Lemma}

\newtheorem{definition}[theorem]{Definition}

\newtheorem{claim}{Claim}

\makeatletter
\let\c@figure\c@theorem
\makeatother

\makeatletter
\@addtoreset{figure}{section}
\makeatother

\title{Levi's Lemma, pseudolinear drawings of $K_n$,\\ and empty triangles}

\author{Alan Arroyo\footnote{Supported by CONACYT.}\ ${}^+$, Dan McQuillan${}^\pm$,\\  R.\ Bruce Richter\footnote{Supported by NSERC.
\newline{\textcolor{white}{...}}
${}^+$University of Waterloo, ${}^\pm$Norwich University, and ${}^\times$UASLP}\ ${}^+$, and Gelasio Salazar${}^*$${}^\times$}

\date{\LaTeX-ed: \today}

\newenvironment{proof}%
{\noindent{\bf Proof.}\ }%
{\hfill\eopf\par\bigskip}%

{\noindent{\bf Proof of #1.}\ }%
{\hfill\eopf\par\bigskip}%

{\noindent{\bf Sketch of a proof.}\ }%
{\hfill\eopf\par\bigskip}%
{\noindent{\bf Ideas for a proof.}\ }%
{\hfill\eopf\par\bigskip}%
{\noindent{\bf Proof in progress.}\ }%
{\hfill\eopf\par\bigskip}%

\def\eop{\hfill{{\rule[0ex]{7pt}{7pt}}}}

\newenvironment{cproofof}[1]
{\bigskip\noindent{\bf Proof of #1.}\startClaims\ }
{\hfill{\eop}\par\bigskip}

\newenvironment{cproof}
{\noindent{\bf Proof.}\startClaims\ }
{\hfill{\eop}\par\bigskip}

\def\i4c{{inter\-nally-4-con\-nec\-ted}}
\def\p4c{\wording{peri\-phe\-rally-4-con\-nec\-ted}}
\def\ei4c{\operatorname{\wording{p4c}}}

\def\2cc{{2-cros\-sing-cri\-tical}}

\def\m2{{{\cal M}_2^3}}

\newcommand{\eopf}{\raisebox{0.8ex}{\framebox{}}}

\newcommand{\majorrem}[1]{}
\newcommand{\minorrem}[1]{}
\newcommand{\wording}[1]{#1}
\newcommand{\wordingrem}[1]{}
\newcommand{\dragominorrem}[1]{}

\def\rtwo{\mathbb R^2}

\def\r5{r^*_{+5}}

\def\rightspine #1{{}_{#1}\kern-3pt\sqsubset}

\def\Side#1#2{\Sigma_{#1}^{#2}}

\newcommand{\change}[1]{{#1}}

\begin{document}

\maketitle

\begin{abstract} There are three main thrusts to this article: a new proof of Levi's Enlargement Lemma for pseudoline arrangements in the real projective plane; a new characterization of pseudolinear drawings of the complete graph; and proofs that pseudolinear and convex drawings of $K_n$ have $n^2+{}$O$(n\log n)$ and O$(n^2)$, respectively, empty triangles.  All the arguments are elementary, algorithmic, and self-contained.
\end{abstract}

AMS Subject Classification Primary 52C30; Secondary 05C10, 68R10
\baselineskip =18pt

\section{Introduction}\label{sec:intro}

The Harary-Hill Conjecture asserts that the crossing number of the complete graph $K_n$ is equal to 
\[
H(n) := \frac 14\left\lfloor\frac{\mathstrut n}{\mathstrut 2}\right\rfloor
\left\lfloor\frac{\mathstrut n-1}{\mathstrut 2}\right\rfloor
\left\lfloor\frac{\mathstrut n-2}{\mathstrut 2}\right\rfloor
\left\lfloor\frac{\mathstrut n-3}{\mathstrut 2}\right\rfloor \,.
\]
The work of \'Abrego et al \cite{abrego} verifies this conjecture for ``shellable" drawings of $K_n$; this is one of the first works that identifies a topological, as opposed to geometric, criterion for a drawing to have at least $H(n)$ crossings.  

Throughout this work, all drawings of graphs are {\em good drawings\/}: no two edges incident with a common vertex cross; no three edges cross at a common point; and no two edges cross each other more than once.

It is well-known that the {\em rectilinear\/} crossing number (all edges are required to be straight-line segments) of $K_n$ is, for $n\ge 10$, strictly larger than $H(n)$.  In fact, this applies to the more general {\em pseudolinear\/} crossing number.  

An {\em arrangement of pseudolines\/} $\Sigma$ is a finite set of simple open arcs in the plane $\rtwo$ such that:  for each $\sigma\in\Sigma$, $\rtwo\setminus \sigma$ is not connected; and for distinct $\sigma$ and $\sigma'$ in $\Sigma$, $\sigma\cap \sigma'$ consists of a single point, which is a crossing.   

A drawing of $K_n$ is {\em pseudolinear\/} if there is an arrangement of $\Sigma$ of pseudolines such that each edge of $K_n$ is contained in one of the pseudolines and each pseudoline contains just one edge.    It is clear that a rectilinear drawing (chosen so no two lines are parallel) is pseudolinear.

The arguments (originally due to Lov\'asz et al \cite{lovasz} and, independently, \'Abrego and Fern\'andez-Merchant \cite{af-m}) that show every rectilinear drawing of $K_n$ has at least $H(n)$ crossings apply equally well to pseudolinear drawings.

The proof that every optimal pseudolinear drawing of $K_n$ has its outer face bounded by a triangle \cite{optTriang} uses the ``allowable sequence" characterization of pseudoline arrangements of Goodman and Pollack \cite{goodPoll}.  Our principal result is that there is another, topological, characterization of pseudolinear drawings of $K_n$.

Let $D$ be a drawing of $K_n$ in the sphere.  For any three distinct vertices $u,v,w$ of $K_n$, the triangle $T$ induced by $u,v,w$ is such that $D[T]$ \change{(the subdrawing of $D$ induced by the subgraph $T$)} is a simple closed curve in the sphere.  

\change {This simple observation leads to the natural ideas of a convex drawing of $K_n$ and a face-convex drawing of $K_n$, which capture at different levels of generality the notion of a convex set in Euclidean space.}

\begin{definition}\label{df:convex} { Let $D$ be a drawing of $K_n$ in the sphere. }
\begin{enumerate} 
\item {Let $T$ be a 3-cycle in $K_n$.  
Then  a closed disc $\Delta$ bounded by $D[T]$ is {\em convex\/} if, for any distinct vertices $u$ and $v$ of $K_n$ such that both $D[u]$ and $D[v]$ are in $\Delta$, then $D[uv]\subseteq \Delta$.}
\item {The drawing $D$  is {\em convex\/} if, for every 3-cycle $T$ in $K_n$, at least one of the closed discs bounded by $D[T]$ is convex.} 
\item \change{A {\em face\/} of $D$ is a component of $\mathbb R^2\setminus D[K_n]$.}
\item\label{it:faceConvex}  { The drawing $D$ is {\em face-convex\/} if there is a face $F$ of $D$ such that, for every triangle $T$ of $D$, the closed disc bounded by $D[T]$ and not containing $F$ is convex.  The face $F$ is the {\em outer face\/} of $D$.}\end{enumerate}
\end{definition}

There seem to be interesting connections between convexity and Knuth's CC systems \cite{knuth}, but we have not yet  formalized this.

In Definition \ref{df:convex} \ref{it:faceConvex}, there is necessarily at least one outer face that shows the drawing to be face-convex.   The unique drawing of $K_6$ with three crossings has two such faces.  (See Figure \ref{fg:TkS}.)  

It is convenient for the definition of convexity to use drawings in the sphere:  every simple closed curve is the boundary of two closed discs.   \change{Every drawing in the plane is converted  by the standard 1-point compactification into a spherical drawing.  Keeping track of the infinite face $F$ in a pseudolinear drawing in the plane results in a face-convex drawing in the sphere with outer face $F$.  The interesting point is the converse: if we convert the face $F$ in the definition of face-convex to be the unbounded face, then the resulting drawing in the plane is pseudolinear.}

 \begin{figure}[!ht]
\begin{center}
\includegraphics[scale=.2]{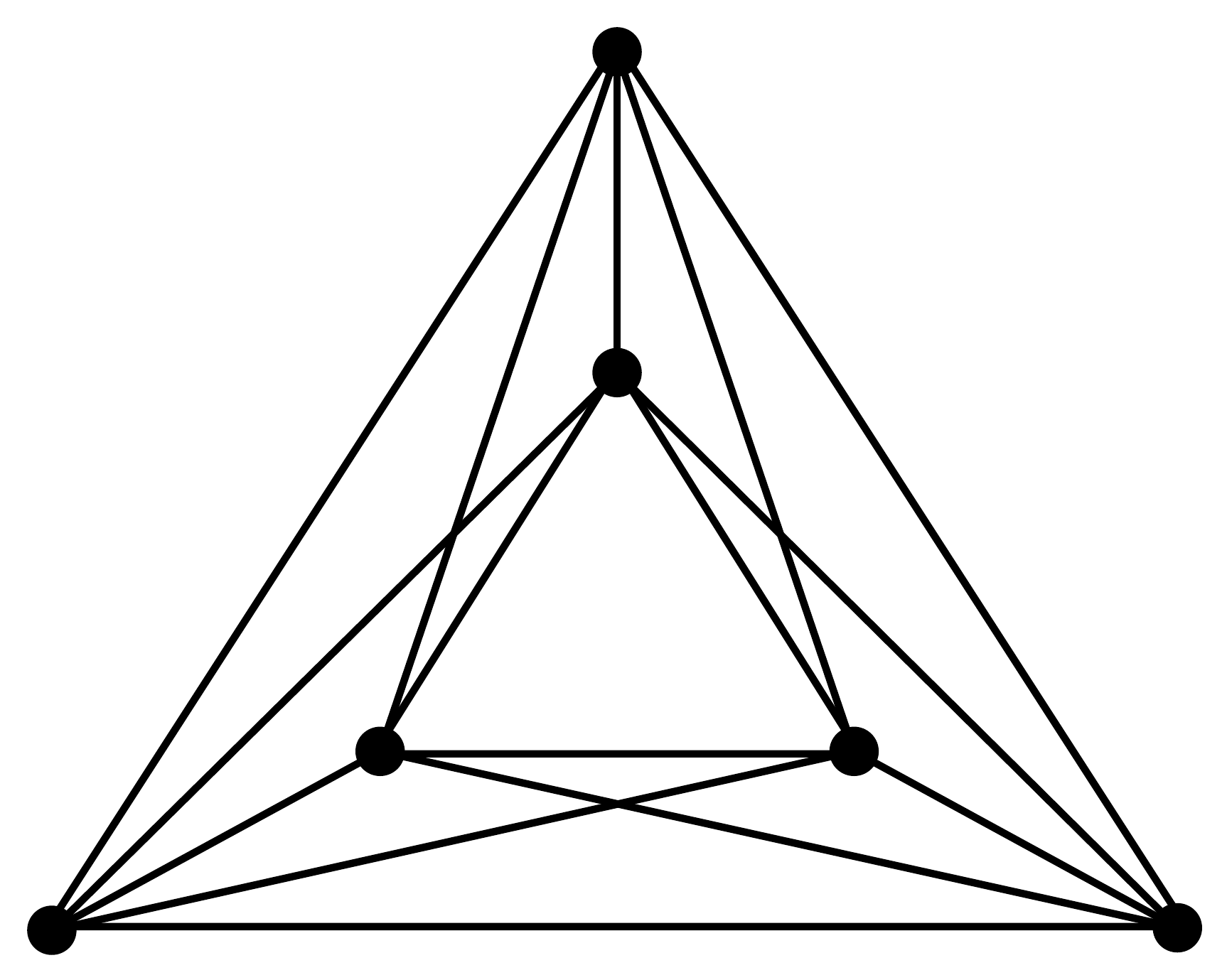}
\caption{The two faces bounded by 3-cycles can each be the outer face.}\label{fg:TkS}
\end{center}
\end{figure}

\begin{theorem}\label{th:strongConvex}  A drawing of $K_n$ in the plane is face-convex if and only if it is pseudolinear.  \end{theorem}

This theorem is proved in Section \ref{sec:pseudolinear}.  An independent recent proof has been found by Aichholzer et al \cite{ahpsv}; their proof uses Knuth's  CC systems \cite{knuth} (reinforcing the interest in the connection with convexity), the duals of which are realizable as pseudolinear arrangements of lines.    Moreover, their statement is in terms of a forbidden configuration.  Properly speaking, their result is of the form, ``there exists a face relative to which the forbidden configuration does not occur".  Their face and our face are the same.  However, our proof is completely different, yielding directly a polynomial time algorithm for finding the pseudolines.

Aichholzer et al show that the there is a pseudolinear drawing of $K_n$ having the same crossing pairs of edges as the given drawing of $K_n$.   Gioan's Theorem \cite{gioan} that any two drawings of $K_n$ with the same crossing pairs of edges are equivalent up to Reidemeister III moves  is then invoked to show that the original drawing is also pseudolinear.     Our proof is completely self-contained; in particular, it does not involve CC-systems and does not invoke Gioan's Theorem.   

The ideas we use are elementary and derive from a simple, direct proof of Levi's Enlargement Lemma given in Section \ref{sec:levy}.  In a separate paper \cite{usGioan}, we give a proof of Gioan's Theorem in the same spirit.

In Section \ref{sec:emptyTriangles}, we extend the B\'ar\'any and F\"uredi \cite{barany} theorem that a rectilinear drawing of $K_n$ has at least $n^2+O(n\log n)$ empty triangles to pseudolinear drawings of $K_n$.  Moreover, we show that a convex drawing of $K_n$ has at least $n^2/3 + O(n)$ empty triangles.

\section{Proof of Levi's Enlargement Lemma}\label{sec:levy}

In this section, we prove Levi's Enlargement Lemma \cite{levy}.  This important fact seems to have only one proof by direct geometric methods in English \cite{grunbaum}.  The proof in \cite{grunbaum} includes a simple step that Gr\"unbaum admits seems clumsy.  Our proof avoids this technicality.  (There is another proof by Sturmfels and Ziegler via oriented matroids \cite{sturmfels}.)

One fact we do use is that there is an alternative definition of an arrangement of pseudolines.  Equivalent to the definition given in the introduction, a {\em pseudoline\/} is a non-contractible simple closed curve in the real projective plane, and an arrangement of pseudolines is a set of pseudolines, any two intersecting in exactly one point; the intersection is necessarily a crossing point.  This perspective will be used in the proof of Levi's Enlargement Lemma.

\begin{theorem}[Levi's Enlargement Lemma]\label{th:levy}
Let $\Sigma$ be an arrangement of pseudolines  and let $a,b$ be any two points in the plane not both in the same pseudoline in $\Sigma$.  Then there is a pseudoline $\sigma$ that contains both $a$ and $b$ and such that $\Sigma\cup\{\sigma\}$ is an arrangement of pseudolines.
\end{theorem}

The principal ingredients in all our arguments are two considerations of the facial structure of an arrangement of pseudolines.   In fact, we need something slightly more general.  An {\em arrangement of arcs\/} is a finite set $\Sigma$ of open arcs in the plane $\mathbb R^2$ such that, for every $\sigma\in\Sigma$, $\rtwo\setminus \sigma$ is not connected and any two elements of $\Sigma$ have at most one point in common, which must be a crossing.   Thus, two arcs in an arrangement of arcs may have no intersection and so be ``parallel".

Let $\Sigma$ be an arrangement of arcs.  Set $\mathcal P(\Sigma)$ to be the set $\bigcup_{\sigma\in\Sigma}\sigma$ of points in the plane.   A {\em face\/} of $\Sigma$ is a component of $\mathbb R^2\setminus \mathcal P(\Sigma)$.
Since $\Sigma$ is finite, there are only finitely many faces of $\Sigma$.

The {\em dual\/} $\Sigma^*$ of $\Sigma$ is the finite graph whose vertices are the faces of $\Sigma$ and there is one edge for each segment $\alpha$ of each $\sigma\in \Sigma$ such that $\alpha$ is one of the components of $\sigma\setminus \mathcal P(\Sigma\setminus\{\sigma\})$.  The dual edge corresponding to $\alpha$ joins the faces of $\Sigma$ on either side of $\alpha$.

Levi's Lemma is a consequence of  our first consideration  of the facial structure of an arrangement of arcs.

\begin{lemma}[Existence of dual paths]\label{lm:dualPaths}  Let $\Sigma$ be an arrangement of arcs and let $a,b$ be points of the plane not in any line in $\Sigma$.  Then there is an $ab$-path in $\Sigma^*$ crossing each arc in $\Sigma$ at most once.
\end{lemma}

\begin{cproof}
We proceed by induction on the number of curves in $\Sigma$ that separate $a$ from $b$, the result being trivial if there are none.  Otherwise, for $x\in \{a,b\}$,  let $F_x$ be the face of $\Sigma$ containing $x$ and let $\sigma\in\Sigma$ be incident with $F_a$ and separate $a$ from $b$.  Then $\Sigma^*$ has an edge $F_aF$ that crosses $\sigma$.  

Let $R$ be the region of $\rtwo\setminus \sigma$ that contains $F_b$ and let $\Sigma'$ be the set $\{\sigma'\cap R\mid \sigma'\in \Sigma,\ \sigma'\cap R\ne\varnothing\}$.  The induction implies there is an $FF_b$-path in $\Sigma'{}^*$.  Together with $F_aF$, we have an $F_aF_b$-path in $\Sigma^*$, as required.
\end{cproof}

We now turn to the proof of Levi's Lemma.

\begin{cproofof}{Theorem \ref{th:levy}}  In this proof, we view the pseudoline arrangement $\Sigma$ as  non-contractible simple closed curves in the real projective plane, any two intersecting exactly once. 

If $a$ is not in any arc in $\Sigma$, then let $F$ be the face of $\Sigma$ containing $a$; replace $a$ with any point in the boundary of $F$ and not in the intersection of two arcs in $\Sigma$.   Likewise, for $b$.  In all cases, the points representing $a$ and $b$ are chosen to be in different arcs in $\Sigma$.

If we find the required $ab$-arc $\sigma$ to extend $\Sigma$ using one or two replacement points, then $\sigma$ goes through the face(s) of $\Sigma$ containing the original point(s), and so we may reroute $\sigma$ to go through the original points, as required.  Thus, we may assume $a$ and $b$ are both in arcs in $\Sigma$.

Let $\Sigma_a$ consist of the arcs in $\Sigma$ containing $a$ and let $F^{(a)}_b$ be the face of $\mathcal P(\Sigma_a)$ containing $b$.  Up to spherical homeomorphisms, there is a unique small arc $\alpha$ through $a$ that has one end in $F^{(a)}_b$ and crosses all the arcs in $\Sigma_a$ at $a$.  The ends of this arc are the two points $a',a''$.    In a similar way, we get the small arc $\beta$ through $b$ joining the two points $b',b''$.  

The choices for $\alpha$ and $\beta$ show that we may label $a',a''$ and $b',b''$ so that $a'$ and $b'$ are in the same face $F'$ of $\Sigma_a\cup \Sigma_b$.    We apply Lemma \ref{lm:dualPaths} to this component to obtain an $a'b'$-arc $\gamma'$ contained in $F'$.

The arc composed of  $\gamma'$ together with the little arcs $\alpha$ and $\beta$ crosses every arc in $\Sigma_a\cup \Sigma_b$ exactly once.  This shows that $a''$ and $b''$ are  in the same face $F''$ of $\Sigma_a\cup \Sigma_b$.

Let $\Sigma''$ be the set $\{\sigma\cap F''\mid \sigma\in \Sigma\, \sigma\cap F''\ne\varnothing\}$.  Lemma \ref{lm:dualPaths} implies there is an $a''b''$-arc $\gamma''$ in $F''$ crossing each element of $\Sigma''$ at most once.

Let $\gamma$ be the closed curve $\gamma'\cup \alpha\cup\gamma''\cup \beta$, adjusted as necessary  near $a'$, $a''$, $b'$, and $b''$ so that $\gamma$ is a  simple closed curve.  It is clear that $\gamma$ crosses each arc in $\Sigma_a\cup \Sigma_b$ exactly once and, therefore, is non-contractible.  By construction, $\gamma$ crosses any arc in $\Sigma$ at most twice; both being non-contractible implies this is in fact at most once.  Therefore, $\Sigma\cup \{\gamma\}$ is the desired arrangement of pseudolines.\end{cproofof}

\section{Proof of Theorem \ref{th:strongConvex}}\label{sec:pseudolinear}

In this section we prove Theorem \ref{th:strongConvex}:  a face-convex drawing of $K_n$ in the sphere with outer face $F$ is a pseudolinear drawing in the plane by making $F$ the infinite face.

It is evident that  face-convexity is inherited in the sense that if $D$ is a face-convex drawing of $K_n$ and $v$ is any vertex of $K_n$, then $D[K_n-v]$ is a face-convex drawing of $K_n-v$.  We begin with a simple observation.

\begin{lemma}\label{lm:noBadK4}
Let $D$ be a face-convex drawing of $K_n$ with outer face $F$.  If $J$ is any $K_4$ in $K_n$ such that $D[J]$ has a crossing, then $F$ is in the face of $D[J]$ bounded by a 4-cycle of $J$.  In particular, no crossing of $D$ is incident with $F$, so $F$ is bounded by a cycle in $K_n$.
\end{lemma}

\begin{cproof}
Let $v,w,x,y$ be the four vertices of $J$ labelled so that $vw$ crosses $xy$ in $D$.  Consider, for example, the triangle $T=(v,w,x,v)$.   {The vertex $y$ is in a closed face $F_y$ of $D[T]$.  Since $xy$ crosses $vw$, $D[xy]$ is not contained in $F_y$, so $F_y$ is not convex.  Since $D$ is face-convex, it follows that $F\subseteq F_y$.}   

Thus, none of $v,w,x,y$ is on the convex side of the triangle containing the other three vertices.  It follows that $F$ is contained in the face of $D[J]$ bounded by the 4-cycle $(v,x,w,y,v)$, as required.  

Inserting a vertex at every crossing point of $D$ produces a 2-connected planar embedding of the resulting graph having $F$ as a face.  This face is bounded by a cycle; since no inserted vertex is incident with $F$, this cycle is a cycle of $K_n$.
\end{cproof}

We remark that Lemma \ref{lm:noBadK4} shows that a face-convex drawing does not have the forbidden configuration of Aichholzer et al \cite{ahpsv}.    The converse is no harder.

\change{For a face-convex drawing $D$ of $K_n$ with outer face $F$, let $C_F$ denote the cycle of $K_n$ bounding  $F$} and let $\Delta_F$ denote the closed disc bounded by $C_F$ and disjoint from $F$.  For any subset $W$ of vertices of $K_n$,  let $D[W]$ denote the subdrawing of $D$ induced by the complete subgraph having precisely the vertices in $W$.   Since $D[W]$ is a face-convex drawing, if $|W|\ge 3$, then its face $F_W$ containing $F$ is bounded by a cycle $C_W$.  The closed disc $\Delta_W$ bounded by $D[C_W]$ and disjoint from $F$ is the {\em convex hull of $W$\/}.   

\ignore{Replacing $W$ with $W\cup \{x,y\}$ is the heart of the simple proof of the following observation.  

\begin{lemma}\label{lm:convexHull}
Let $D$ be a face-convex drawing of $K_n$ with outer face $F$ and let $W$ be a subset of $V(K_n)$.  If $|W|\ge 3$ and $x$ any $y$ are any vertices in of $K_n$ in $\Delta_W$, then $D[xy]\subseteq \Delta_W$.      \hfill\eop
\end{lemma}
}

For each edge $uv$ of $G$, $D[uv]$ is a simple arc in the sphere.  Arbitrarily giving $D[uv]$ a direction distinguishes a left and right side to the arc $D[uv]$.  We prefer not to use the labels `left' and `right', as we find them somewhat confusing. 
For now, we shall arbitrarily label them as {\em side 1\/} and {\em side 2\/}  of $uv$.

For each vertex $w$ different from $u$ and $v$, {\em $w$ is on side $i$ of $uv$\/} if the face of $D[\{u,v,w\}]$ disjoint from $F$ is on side $i$ of $uv$.   We set $\Side{uv}i$ to be the set of vertices on side $i$ of $uv$; \change{for convenience, we} include $u$ and $v$ in $\Side{uv}i$.

It is clear that $\Side{uv}1\cap \Side{uv}2=\{u,v\}$.  What is less clear is that $D[\Side{uv}1]\cap D[\Side{uv}2]$ consists just of $u$, $v$, and $uv$.  The next lemma is a useful step in proving this.

\begin{lemma}\label{lm:K4K5sameSides}
Let $D$ be a face-convex drawing of $K_n$ \change
{with outer face $F$}  and let $u,v,x,y$ be distinct vertices of $K_n$.  
\begin{enumerate}[label={\bf(\ref{lm:K4K5sameSides}.\arabic*)},ref=(\ref{lm:K4K5sameSides}.\arabic*),leftmargin=60pt]
\item\label{it:K4sameSide} Then $x$ and $y$ are on the same side of $uv$ if and only if $uv$ is incident with $F_{\{u,v,x,y\}}$. 
\item\label{it:differentSidesNoCrossing}  In particular, if $x$ and $y$ are on different sides of $uv$, then $D[\{u,v,x,y\}]-xy$ has no crossing.
\item\label{it:interiorDifferentSides} If $z$ is any vertex such that $u$ is in the interior of $\Delta_{\{x,y,z\}}$, then some two of $x$, $y$, and $z$ are on different sides of $uv$.  
\end{enumerate}
\end{lemma}

\begin{cproof}  Ultimately, the easiest way to understand \ref{it:K4sameSide} is to draw the two possible drawings of $K_4$ and{, in both cases,}  check the two possibilities:  $uv$ is incident with $F_{\{u,v,x,y\}}$ and $uv$ is not incident with $F_{\{u,v,x,y\}}$.  In the case the $K_4$ has a crossing, $F_{\{u,v,x,y\}}$ is the face bounded by the 4-cycle.

For \ref{it:differentSidesNoCrossing}, let $J$ be the $K_4$ induced by $u,v,x,y$.  Since $x$ and $y$ are on different sides of $uv$, the preceding conclusion shows that either $D[J]$ has no crossing, in which case we are done, or $uv$ is crossed in $D[J]$ and it crosses $xy$.   As this is the only crossing in $D[J]$, $D[J]-xy$ has no crossing.

Finally, we consider \ref{it:interiorDifferentSides}.  If $D[v]\notin \Delta_{\{x,y,z\}}$, then $D[uv]$ crosses the 3-cycle $xyz$.  Now \ref{it:K4sameSide} shows that the ends of the edge crossing $D[uv]$ are on different sides of $uv$.  Thus, we may assume $D[v]\in \Delta_{\{x,y,z\}}$.

Since $D[u]\subseteq \Delta_{\{x,y,z\}}$, $D[\{ux,uy,uz\}]\subseteq \Delta_{\{x,y,z\}}$.  If $v=z$, then $D[uv]$ is not incident with $F_{\{u,v,x,y\}}$.  Therefore, \ref{it:K4sameSide} shows $x$ and $y$ are on different sides of $uv$ and consequently, we may assume $v\ne z$.

By definition, $\Delta_{\{u,x,y\}}\subseteq \Delta_{\{x,y,z\}}$, and likewise for $\Delta_{\{u,x,z\}}$ and $\Delta_{\{u,y,z\}}$.   We may choose the labelling of $x$, $y$, and $z$ so that $D[v]\in \Delta_{\{u,x,y\}}$.  But now $uv$ is not in the boundary of $D[\{u,v,x,y\}]$.  Again, \ref{it:K4sameSide} shows $x$ and $y$ are on different sides of $uv$.
\end{cproof}

We are now ready for the first significant step, which is Item \ref{it:noDiscIntersection} in our next result.

\begin{lemma}\label{lm:basicProperties}
Let $D$ be a face-convex drawing of $K_n$ \change{with outer face $F$}, let $W\subseteq V(K_n)$, and let $uv$ be any edge of $K_n$.  
\begin{enumerate}[label={\bf(\ref{lm:basicProperties}.\arabic*)},ref=(\ref{lm:basicProperties}.\arabic*),leftmargin=60pt]
\item\label{it:emptySide}  Both $(W\cap \Side{uv}1)\setminus\{u,v\}$ and $(W\cap \Side{uv}2)\setminus\{u,v\}$ are not empty if and only if $uv$ is not incident with $F_{W\cup \{u,v\}}$.  \item\label{it:oppSideNotInDelta}  For $\{i,j\}=\{1,2\}$,  no vertex of $ \Side{uv}i\setminus\{u,v\}$ is in  $\Delta_{(W\cup \{u,v\})\cap \Side{uv}j}$.
\item\label{it:noSidesCrossings} If, for $i=1,2$, $x_i,y_i\in\Side{uv}i$, then $x_1y_1$ does not cross $x_2y_2$ in $D$.  
\item\label{it:noDiscIntersection} $\Delta_{\Side{uv}1}\cap \Delta_{\Side{uv}2}=D[\{u,v\}]$.
\end{enumerate} 
\end{lemma}

\begin{cproof}  Suppose $uv$ is incident with $F_{W\cup\{u,v\}}$.  For any $x,y\in W\setminus\{u,v\}$, it follows that $uv$ is incident with $F_{\{u,v,x,y\}}$, so Lemma \ref{it:K4sameSide} shows $x$ and $y$ are on the same side of $uv$.

Conversely, suppose all vertices in $W\setminus\{u,v\}$ are on the same side of $uv$.  The closed disc $\Delta_{W\cup\{u,v\}}$ is the union of all the convex sides $\Delta_{\{x,y,z\}}$, for $x,y,z\in W\cup \{u,v\}$.  If $u$ is in the interior of some $\Delta_{\{x,y,z\}}$, then Lemma \ref{it:interiorDifferentSides} shows some two of $x$ and $y$ are on different sides of $uv$.  Thus, both $u$ and $v$ are in $C_{W\cup \{u,v\}}$. If $uv\notin E(C_{W\cup \{u,v\}})$, then $C_{W\cup \{u,v\}}-\{u,v\}$ is not connected; let $x$ and $y$ be in different components of $C_{W\cup \{u,v\}}-\{u,v\}$.  Then $D[xy]$ crosses $D[uv]$, showing $x$ and $y$ are on different sides of $uv$.  This contradiction completes the proof of \ref{it:emptySide}.

For $i=1,2$, let $W_i=(W\cup \{u,v\})\cap \Side{uv}i$. 

{ For \ref{it:oppSideNotInDelta}, suppose $x\in  \Side{uv}i\setminus\{u,v\}$ is in $\Delta_{W_j}$.  Since $W_j\subseteq W_j\cup \{x\}$, $\Delta_{W_j}\subseteq \Delta_{W_j\cup \{x\}}$.    Since $C_{W_j\cup \{x\}}$ either contains $x$, in which case $D[x]\notin \Delta_{W_j}$,  or is $C_{W_j}$, in which case $D[x]$ is in the interior of $\Delta_{W_j}$.  

Assume by way of contradiction that it is the latter case.  Then $C_{W_j\cup \{x\}}=C_{W_j}$.  Therefore, $\Delta_{W_j\cup \{x\}}=\Delta_{W_j}$.  Since $uv$ is in $C_{W_j}$ by \ref{it:emptySide}, the other direction of \ref{it:emptySide} implies the contradiction that $x\in \Side{uv}j$.  }

For \ref{it:noSidesCrossings}, we suppose $x_1y_1$ and $x_2y_2$ cross in $D$.  From \ref{it:differentSidesNoCrossing},  not both $\{x_1,y_1\}$ and $\{x_2,y_2\}$ can contain an element of $\{u,v\}$.  We may choose the labelling so that $\{x_1,y_1\}\cap \{u,v\}=\varnothing$ and let $J_1$ be the $K_4$ induced by $u,v,x_1,y_1$.  Since $\{x_2,y_2\}\ne \{u,v\}$, we may assume $x_2\notin\{u,v\}$.

\begin{claim}\label{cl:y2NotUOrV} $y_2\notin \{u,v\}$.  
\end{claim}

\begin{proof}  Suppose $y_2\in \{u,v\}$.  Apply \ref{it:differentSidesNoCrossing} to each of the $K_4$'s induced by $u,v,x_2,x_1$ and $u,v,x_2,y_1$.  The conclusion is that  $x_2y_2$ does not cross $D[J_1]-x_1y_1$.  Thus, as we follow $D[x_2y_2]$ from $D[x_2]$, its first and only intersection with $D[J_1]$ is with $D[x_1y_1]$, showing $x_1y_1$ is incident with the face $F_{J_1}$.  Since $uv$ is also incident with $F_{J_1}$, we deduce that $D[J_1]$ is a crossing $K_4$.

However, continuing on to the end $y_2\in\{u,v\}$, $D[x_2y_2]$ must cross $C_{J_1}$ without crossing any edge of $J_1$, which is impossible, as required.   \end{proof}

Let  $J_2$ be the $K_4$ induced by $u,v,x_2,y_2$.     Lemma \ref{it:differentSidesNoCrossing} and Claim 1 show that the only possible crossing between $D[J_1]$ and $D[J_2]$ is the crossing of $x_1y_1$ with $x_2y_2$.  However, both $x_2$ and $y_2$ are in $F_{J_1}$, showing that $x_2y_2$ must cross $C_{J_1}$ an even number of times.  As there is at least one crossing and all the crossings are with $x_1y_1$, we violate the requirement that, in a drawing, no two edges cross more than once.

Now for \ref{it:noDiscIntersection}.  From \ref{it:emptySide}, no vertex of one side is inside the convex hull of the other side.  
Going one step further, suppose $x,y\in (W\cup \{u,v\})\cap \Side{uv} 2$ is such that $D[xy]$ has a point that is in $\Delta_{(W\cup \{u,v\})\cap \Side{uv}1}$.  Then $xy$ crosses some edge of $C_{\Side{uv}1}$, contradicting \ref{it:noSidesCrossings}.  

Finally, we show that $\Delta_{\Side{uv}1}\cap \Delta_{\Side{uv}2}=D[\{u,v\}]$.  The cycles $C_{\Side{uv}1}$ and $C_{\Side{uv}2}$ are disjoint except for $uv$.  If there is some  point $a$ of the sphere in $\Delta_{\Side{uv}1}\cap \Delta_{\Side{uv}2}$  that is not in $uv$, then $a$ is in the convex hull of both $C_{\Side{uv}1}$ and $C_{\Side{uv}2}$.  This implies that either $\Delta_{\Side{uv}1}\subseteq \Delta_{\Side{uv}2}$ or $\Delta_{\Side{uv}2}\subseteq \Delta_{\Side{uv}1}$, contradicting  \ref{it:emptySide}.
  \end{cproof}

It follows from the above that, for every edge $uv$, $\Delta_{\Side{uv}1}\cup \Delta_{\Side{uv}2}$ includes all the vertices of $K_n$ and all edges that have both ends in the same one of  $\Side{uv}1$ and $\Side{uv}2$.  We obtain a more refined understanding of the relationship of this subdrawing with the entire drawing in the following.

\begin{lemma}\label{lm:noneOrTwo}
Let $D$ be a face-convex drawing of $K_n$ \change{with outer face $F$} and let $uv$ be any edge of $K_n$.  Let $W$ be any subset of $V(K_n)$.  Then there are not four distinct vertices $x_1,x_2,y_1,y_2$ of $C_W$ appearing in this cyclic order in $C_W$ such that, for $i=1,2$, $x_i,y_i\in \Side{uv}i$.
\end{lemma}

\begin{cproof}
If such vertices exist, then the edges $x_1y_1$ and $x_2y_2$ are both in $\Delta_{W}$ and they cross, contradicting Lemma \ref{it:noSidesCrossings}.
\end{cproof}

It follows from Lemma \ref{lm:noneOrTwo} that, for a face-convex drawing of $K_n$ with outer face $F$,  $C_F$ has, for $i=1,2$, a path (possibly with no vertices; this happens only when $uv$ is in $C_F$) contained in $\Side{uv}i\setminus\{u,v\}$.  The ends of these paths are connected in $C_F$ either by an edge or by a path of length 2, the middle vertex being one of $u$ and $v$.

Henceforth, we assume $uv\notin E(C_F)$; that is, we assume both $\Side{uv}1\setminus\{u,v\}$ and $\Side{uv}2\setminus\{u,v\}$ are both non-empty.  In this case,  $C_F\cup C_{\Side{uv}1}\cup C_{\Side{uv}2}$ is a planar embedding of a 2-connected graph.  Three of its faces are $F$, $\Delta_{\Side{uv}1}$, and $\Delta_{\Side{uv}2}$.  { The other faces, if any, are determined by whether or not $u$ or $v$ is in $C_F$. }

Let $A_{uv}$ consist of the ones of $u$ and $v$ not in $C_F$.  For each $a\in A_{uv}$, there is a face $F_{uv}^a$ of $C_F\cup C_{\Side{uv}1}\cup C_{\Side{uv}2}$ incident with $a$ and an edge $f_{uv}^a$ of $C_F$; the edge $f_{uv}^a$ is incident with a vertex in $\Side{uv}1$ and a vertex in $\Side{uv}2$.  Let $Q_{uv}^a$ be the cycle bounding $F_{uv}^a$.  
See Figure \ref{fg:definitionOfA}.

For $j=1,2$, $C_{\Side{uv}j}$ is the union of two internally disjoint paths, namely $C_F\cap C_{\Side{uv}j}$ and the path $P_{uv}^j$ in $C_{\Side{uv}j}$ having its ends in $C_F$ but otherwise disjoint from $C_F$.     If $a$ is in $\{u,v\}\setminus A_{uv}$, then $a$ is in $C_F$ and, therefore, is an end of both $P_{uv}^1$ and $P_{uv}^2$.    If $a\in A_{uv}$, then one end of $f_{uv}^a$ is an end of $P_{uv}^1$ and the other end of $f_{uv}^a$ is an end of $P_{uv}^2$.

\begin{figure}[!ht]
\begin{center}
\scalebox{1.0}{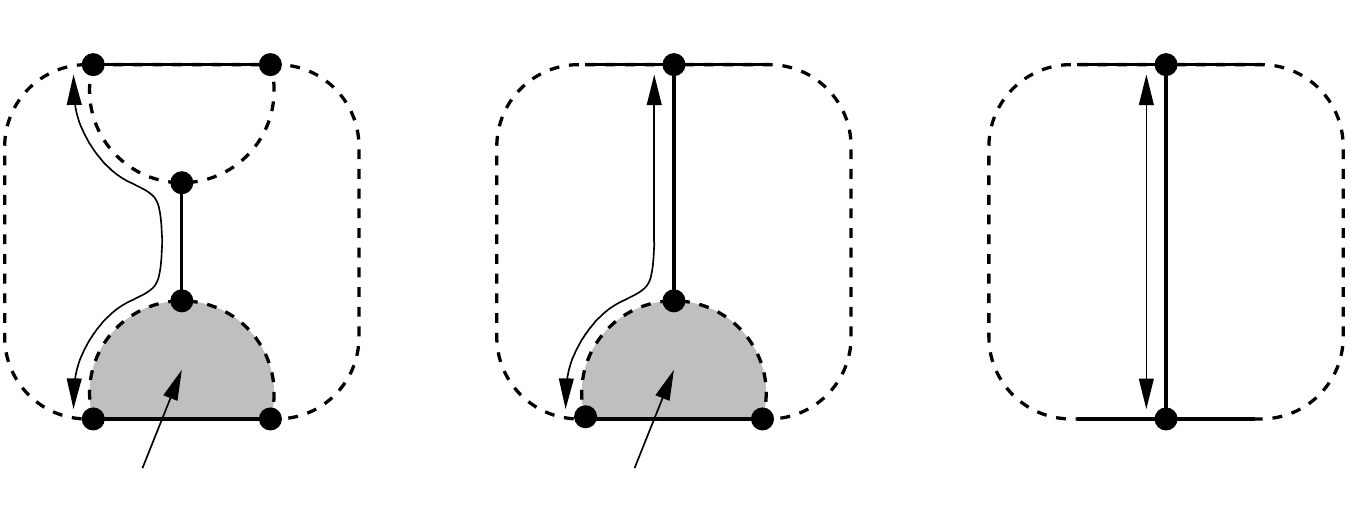}
\caption{In the left-hand figure, $A_e=\{u,v\}$, in the middle $A_e=\{v\}$, and in the right $A_e=\varnothing$.}\label{fg:definitionOfA}
\end{center}
\end{figure}

\ignore{\begin{figure}[!ht]
\begin{center}
\includegraphics[scale=.3]{definitionOfA}
\caption{In the left-hand figure, $A_e=\{u,v\}$, in the middle $A_e=\{v\}$, and in the right $A_e=\varnothing$.}\label{fg:definitionOfA}
\end{center}
\end{figure}
}

We are now prepared to prove our characterization of pseudolinear drawings.  Recall that $\Delta_F$ is the closed disc bounded by $C_F$ that is disjoint from $F$.

\begin{cproofof}{Theorem \ref{th:strongConvex}}
We begin by finding, for each edge $e$ that is not in $C_F$,  an arc $\alpha_{e}$  such that:
\begin{enumerate}[label=(\Roman*)]
\item\label{it:defAlpha} $\alpha_{e}$ consists of three parts, namely $D[e]$, and, for each vertex $u$ incident with $e$, a subarc $\alpha_e^u$, which is either just $u$, if $u\in V(C_F)$, or an arc in $F_e^u$ joining $u$ to a point in $f_e^a$ and otherwise disjoint from $Q_e^u$;
\item\label{it:edgeCrossings} if  $\alpha_{e}$ crosses an edge $e'$ (including possibly $e'\in E(C_F)$), then $e'$ has an incident vertex in each of $\Side{e}1$ and $\Side{e}2$; and 
\item\label{it:twoAlpha} for any other edge $e'$ not in $C_F$, $\alpha_{e}$ and $\alpha_{e'}$ intersect at most once, and if they intersect, the intersection is a crossing point. 
\end{enumerate}

Arbitrarily order the edges of $K_n$ not in $C_F$ as $e_1,\dots,e_r$.  We suppose $i\ge 1$ and we have $\alpha_{e_1},\dots,\alpha_{e_{i-1}}$ satisfying Items \ref{it:defAlpha} -- \ref{it:twoAlpha}.  We show there is an arc $\alpha_{e_i}$ such that $\alpha_{e_1},\dots,\alpha_{e_{i}}$ also satisfy Items \ref{it:defAlpha} -- \ref{it:twoAlpha}.    Let $e_i=uv$.

Since $e_i$ is not in $C_F$, $D[e_i]$ is inside $\Delta_F$.  In $C_F$ there are vertices on each side of $e_i$.     

\medskip\noindent{\bf Useful Fact:}  {\em 
Let $j\in \{1,2,\dots,i-1\}$.  By \ref{it:edgeCrossings} and Lemma \ref{lm:noneOrTwo}, $\alpha_{e_j}$ crosses each of $P_{e_i}^1$ and $P_{e_i}^2$ at most twice.}

\medskip

Since part of the extension of $D[e_i]$ to $\alpha_{e_i}$ is trivial if either $u$ or $v$ is in $C_F$, we will generally proceed below as though neither $u$ nor $v$ is in $C_F$.   When there is a subtlety in the event $u$ or $v$ is in $C_F$, we will specifically mention it. 

We can apply Lemma \ref{lm:dualPaths} in the interior of $F_{e_i}^{u}$ and $F_{e_i}^{v}$ to extend $e_i$ in both directions to points on (actually very near) $f_{e_i}^{u}$ and $f_{e_i}^{v}$ to create a possible $\alpha_{e_i}$.  These are all equivalent up to Reidemeister moves and any one is a potential solution.  We let $\Lambda_i$ denote the set of these dual path solutions.

For $j=1,2,\dots,i-1$, the segment $\alpha_{e_j}$ is {\em unavoidable for $e_i$\/} if $\alpha_{e_j}$ crosses both the paths $P_{e_i}^1$ and $P_{e_i}^2$.  In particular, $\alpha_{e_j}$ is unavoidable if it crosses $e_i$. 

It may be that $e_j$ is incident with one of $u$ and $v$, for example.  As this forces a crossing of $\alpha_{e_j}$ with $\alpha_{e_i}$, we take this as a crossing of both $P_{e_i}^1$ or $P_{e_i}^2$.  On the other hand, if $e_j$ is incident with an end $w$ of $f_{e_i}^u$, then this constitutes a crossing of $\alpha_{e_j}$ with the one of $P_{e_i}^1$ and $P_{e_i}^2$ that contains $w$.

\begin{claim}\label{cl:unavoidable} For $j\in \{1,2,\dots,i-1\}$,  $\alpha_{e_j}$ is unavoidable for $e_i$ if and only if every arc in $\Lambda_i$ crosses $\alpha_{e_j}$.
\end{claim}

\begin{proof}  Suppose first that $\alpha_{e_j}$ is unavoidable for $e_i$.  If $\alpha_{e_j}$ has a point in $D[e_i]$, then evidently $\alpha_{e_j}$ crosses every solution in $\Lambda_i$.

In the case $\alpha_{e_j}$ is disjoint from the closed arc $D[e_i]$, there is some subarc of $\alpha_{e_j}$ with an end in each of $P_{e_i}^1$ and $P_{e_i}^2$, but otherwise disjoint from $P_{e_i}^1\cup P_{e_i}^2$.  This arc must join two points in either $Q_{e_i}^u$ or $Q_{e_i}^v$.   It is clear that every solution in $\Lambda_i$ must cross this arc, as required.

Conversely, if $\alpha_{e_j}$ is not unavoidable, then it does not cross, say, $P_{e_i}^1$.  In this case, there is a solution in $\Lambda_i$ whose extensions of $e_i$ go just inside $F_{e_i}^{u}$ and $F_{e_i}^{v}$, in both cases very close to $P_{e_i}^1$.  This solution does not cross $\alpha_{e_j}$. \end{proof}

Suppose that $\alpha_{e_j}$ is unavoidable for $e_i$ and suppose there is an end $a_j$ of $\alpha_{e_j}$ in $f_{e_i}^{u}$. Following $\alpha_{e_j}$ from $a_j$, we come to a crossing of, say $P_{e_i}^1\cap Q_{e_i}^u$.   The segment of $f_{e_i}^{u}$ from its end $u_{e_i}^1$ in $P_{e_i}^1$ to $a_j$ is {\em restricted for $\alpha_{e_i}$\/}.  We do not want $\alpha_{e_i}$ to cross $\alpha_{e_j}$ on this end segment of $\alpha_{e_j}$, since they must cross elsewhere.  

 It may be that the portion of $\alpha_{e_j}$ from $a_j$ to its first intersection in $P_{e_i}^1\cup P_{e_i}^2$ meets $P_{e_i}^1\cup P_{e_i}^2$ at $u$.  In particular, $u$ is an end of $e_j$.   In this case, it is not immediately clear what the restriction should be.  The other end of $e_j$ is either in $\Side{e_i}1$ or $\Side{e_i}2$, so, correspondingly, $D[e_j]\subseteq \Delta_{\Side{e_i}1}$ or $D[e_j]\subseteq \Delta_{\Side{e_i}2}$.  As $\alpha_{e_i}$ must be made to cross $\alpha_{e_j}$ at their intersection $u$, only in the case $D[e_j]\subseteq \Delta_{\Side{e_i}2}$ do we get a restriction between $u_{e_i}^1$ and $a_j$.    (In the other case, as in the next paragraph, the restriction is between $u_{e_i}^2$ and $a_j$.)

There is a completely analogous restriction from $a_j$ to the other end $u_{e_i}^2$ of $f_{e_i}^{u}$ in $P_{e_i}^2$ if, traversing $\alpha_{e_j}$ from $a_j$, $\alpha_{e_j}$ first meets $P_{e_i}^2$.

Let $R_{u}^1$ be the union of all the $e_j$-restricted portions, $j=1,2,\dots,i-1$, of $f_{e_i}^{u}$ that contain the end $u_{e_i}^1$ of $f_{e_i}^{u}$ and let $R_{u}^2$ be the union of all the restricted portions of $f_{e_i}^{u}$ that contain the other end $u_{e_i}^2$ of $f_{e_i}^{u}$.  

If $u$ is in $C_F$, then the $u$ portion of $\alpha_{e_i}$ is just $u$ and no extension at this end is required.  The restrictions are required in the case $u$ is not in $C_F$, the subject of the next claim.

\begin{claim}  If $u$ is not in $C_F$, then $R_{u}^1\cap R_{u}^2=\varnothing$. \end{claim}

\begin{proof} If the intersection is not empty, then there exist $j,j'\in \{1,2,\dots,i-1\}$ such that:
\begin{enumerate}\item $\alpha_{e_j}$ proceeds from $a_j$ in $f_{e_i}^u$ to $P_{e_i}^1$; \item $\alpha_{e_{j'}}$ proceeds from $a_{j'}$ in $F_{e_i}^u$ to $P_{e_i}^2$; and, \item in $f_{e_i}^u$, $a_{j'}$ is not further from $u_{e_i}^1$ than $a_j$ is.
\end{enumerate}

In particular, $\sigma_{u,j}^1$ and $\sigma_{u,j'}^2$ must cross in $F_{e_i}^{u}$, so they never cross again.  (This is true even if the crossing is $a_j=a_{j'}$.  It turns out that $a_j=a_{j'}$ does not occur in our construction, but we do not need this fact, so we do not use it.)

As we traverse $\alpha_{e_j}$ beginning at $a_j$, we first cross $P_{e_i}^1$ at $\times_{j,1}^1$ in $P_{e_i}^1\cap Q_{e_i}^u$.   
Since $\alpha_j$ is unavoidable, it must cross $P_{e_i}^2$ for the first time at the point $\times_{j,1}^2$.   Between $\times_{j,1}^1$ and $\times_{j,1}^2$, there is a second crossing $\times_{j,2}^1$ with $P_{e_i}^1$; possibly $\times_{j,2}^1=\times_{j,1}^2$.  The Useful Fact implies these are no other crossings of $\alpha_{e_j}$ with $P_{e_i}^1$.  (In $\times^k_{(r,s)}$, the exponent $k$ refers to which $P_{e_i}^k$ is being crossed; the subscripts $(r,s)$ are indicating which arc $\alpha_{e_r}$  is under consideration and, for $s\in\{1,2\}$, it is the $s^{\textrm{th}}$ crossing of $\alpha_{e_r}$  with $P_{e_i}^k$ as we traverse $\alpha_{e_r}$  from $a_r$.)

We claim that the second crossing $\times_{j,2}^1$ of $\alpha_j$ with $P_{e_i}^1$ cannot be in the segment of $P_{e_i}^1$ between $u_{e_i}^1$ and $\times_{j,1}^1$.   To see this,  suppose $\times_{j,2}^1$ is in this segment; let $\sigma_j$ be the segment of $\alpha_{e_j}$ from $\times_{j,2}^1$ to the other end $a'_j$.  The Useful Fact and the non-self-crossing of $\alpha_{e_j}$ imply that $\sigma_j$ is trapped inside the subregion of $F_{e_i}^u$ incident with $u_{e_i}^1$ and the segment of $\alpha_{e_j}$ from $a_j$ to $\times_{j,1}^1$.   The only place $a'_j$ can be is in $f_{e_i}^u$, contradicting \ref{it:defAlpha}.

A very similar argument shows that $\alpha_{e_{j'}}$ cannot cross that same segment of $P_{e_i}^2$. (Such a crossing would be the second of $\alpha_{e_{j'}}$ with $P_{e_i}^2$.  Thus, the other end $a'_{j'}$ of $\alpha_{e_{j'}}$ would also be in $f_{e_i}^u$.)

We conclude that $\times_{j,2}^1$ is in $P_{e_i}^1$ between $\times_{j,1}^1$ and the other end $v_{e_i}^1$ of $P_{e_i}^1$.  

Since $\alpha_{e_{j'}}$ is unavoidable, as we traverse it from $a_{j'}$ in $f_{e_i}^u$, there is a first crossing $\times_{j',1}^1$ of $\alpha_{e_{j'}}$ with $P_{e_i}^1$.  Between $a_{j'}$ and $\times_{j',1}^1$, there are the two crossings $\times_{j',1}^2$ and $\times_{j',2}^2$ of $\alpha_{e_{j'}}$ with $P_{e_i}^2$; possibly $\times_{j',2}^2=\times_{j',1}^1$.
Note that the Useful Fact implies $\alpha_{e_{j'}}$ is disjoint from $C_F\cap C_{\Side{e_i}2}$.

Let $\gamma$ be the simple closed curve consisting of the portion of  $\alpha_{e_j}$ from $a_j$ to $\times_{j,1}^2$, and then the portion of $C_{\Side{uv}2}$ from $\times_{j,1}^2$ to $v_{e_i}^2$ and along $C_F\cap C_{\Side{uv}2}$ to $u_{e_i}^2$, and then the portion of $f_{e_i}^u$ from $u_{e_i}^2$ back to $a_j$.  

From $\times_{j',1}^2$ to the other end $a_{j'}$, $\alpha_{e_{j'}}$ must cross $\gamma$.  The only segment it can cross is the portion of $P_{e_i}^2$ between $\times_{j,1}^2$ and $v_{e_i}^2$.  This implies that 
$\times_{j',2}^2$ is between $\times_{j,1}^2$ and $v_{e_i}^2$ in $P_{e_i}^2$.

Reversing the roles of $j$ and $j'$ and of sides 1 and 2, we conclude that the preceding argument shows that $\times_{j,2}^1$ is between $\times_{j',1}^1$ and $v_{e_i}^1$ in $P_{e_i}^1$.

The simple closed curve $\gamma$ above is crossed by $\alpha_{e_{j'}}$ at the point $\times_{j',2}^2$ in the segment of $P_{e_i}^2$ between $\times_{j,1}^2$ and $v_{e_i}^2$. See Figure \ref{fg:overlapRestrictions}. On the other hand, $\times_{j',1}^1$ is on the segment of $P_{e_i}^1$ between the two points $\times_{j,1}^1$ and $\times_{j,2}^1$ and so is on the other side of $\gamma$.  This shows that $\alpha_{e_{j'}}$ must cross $\gamma$ again and this is impossible.  \end{proof}

\begin{figure}[!ht]
\begin{center}
\scalebox{1.6}{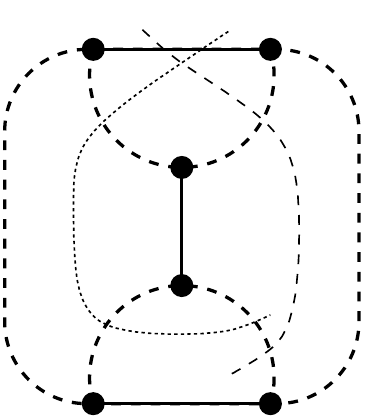}
\caption{One instance of overlapping restrictions.  There is no way for $\alpha_{e_{j'}}$ to get to $\times^1_{j',1}$.}\label{fg:overlapRestrictions}
\end{center}
\end{figure}

\ignore{\begin{figure}[!ht]
\begin{center}
\includegraphics[scale=.5]{OverlapRestrictions}
\caption{One instance of overlapping restrictions.  There is no way for $\alpha_{e_{j'}}$ to get to $\times^1_{j',1}$.}\label{fg:overlapRestrictions}
\end{center}
\end{figure}
}

If $f_{e_i}^{u}$ does not exist, then set $\eta_{u}=u$.  Otherwise, let $\rho_{u}$ be either $u_{e_i}^1$ or the point of $R_{u}^1$ furthest from $u_{e_i}^1$.  Then $\eta_{u}$ is any point between $\rho_{u}$ and the next point between $\rho_{u}$ and $u_{e_i}^2$ that is an end of some $\alpha_{j}$, for $j\in \{1,2,\dots,i-1\}$.  Likewise,  $\eta_{v}$ is any point of $f_{e_i}^v$ between the last point $\rho_v$ of $R_v^2$ and the next point between $\rho_u$ and $v_{e_1}^1$ that is an end of some $\alpha_j$, for $j\in \{1,2,\dots,i-1\}$.  (Notice that we use the $P_{e_i}^1$-side restrictions at the ``$u$-end" and the $P_{e_i}^2$-side restrictions at the $v$-end.    We could have equally well used the $P_{e_i}^2$-side restrictions at the $u$-end and the $P_{e_i}^1$-restrictions at the $v$-end.)

We now apply Lemma \ref{lm:dualPaths} to the region $F_{e_i}^{u}$ (if it exists) using $u$ and $\eta_{u}$ as ends to be connected.  We do the same thing in $F_{e_i}^{v}$ joining $v$ and $\eta_{v}$.  These arcs together with $D[e_i]$ give us $\alpha_{e_i}$, as described in \ref{it:defAlpha}.  (See Figure \ref{fg:alphaEi}.)

\begin{figure}[!ht]
\begin{center}
\scalebox{1.6}{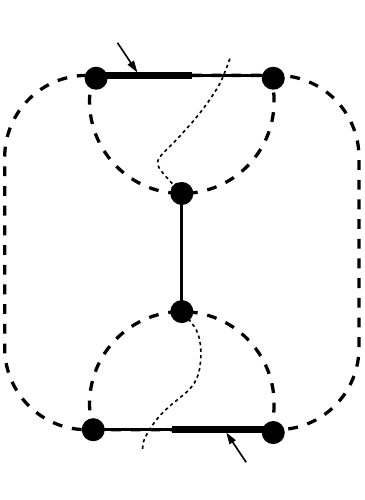}
\caption{The arc $\alpha_{e_i}$.}\label{fg:alphaEi}
\end{center}
\end{figure}

\ignore{
\begin{figure}[!ht]
\begin{center}
\includegraphics[scale=.5]{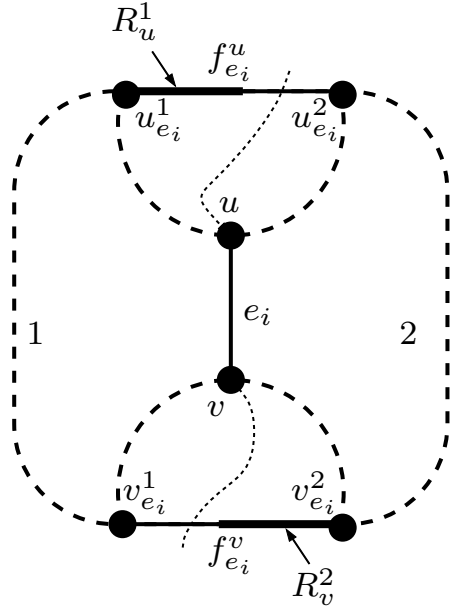}
\caption{The arc $\alpha_{e_i}$.}\label{fg:alphaEi}
\end{center}
\end{figure}
}

The construction of $\alpha_{e_i}$ makes it clear that $\alpha_{e_i}$ meets each of $\Delta_{\Side{e_i}1}$ and $\Delta_{\Side{e_i}2}$ in $e_i$.  Therefore, $\alpha_{e_i}$ satisfies \ref{it:edgeCrossings}.

\begin{claim}\label{cl:unavoidableOnce}  For any $j\in \{1,2,\dots,i-1\}$ for which $\alpha_{e_j}$ is unavoidable,  $\alpha_{e_i}$ crosses $\alpha_{e_j}$ exactly once. 
\end{claim}

\begin{proof}  Because of the restrictions, $\alpha_{e_i}$ does not cross the portions (if either or both of these exist) of $\alpha_{e_j}$ from $f_{e_i}^u$ to its first intersection with $P_{e_i}^1\cup P_{e_i}^2$ and the analogous segment from $f_{e_i}^v$.  It is enough to show that $\alpha_{e_j}$ does not have two  completely disjoint segments that have one end in $P_{e_i}^1$ and one end in $P_{e_i}^2$, but otherwise disjoint from $P_{e_i}^1\cup P_{e_i}^2$; this includes the possibility of one segment consisting of just one point in $e_i$.

Suppose $\tau_1$ and $\tau_2$ are two such segments of $\alpha_j$.  Since each involves a crossing of each of $P_{e_i}^1$ and $P_{e_i}^2$, the Useful Fact implies that these are the only crossings of $\alpha_j$ with $P_{e_i}^1\cup P_{e_i}^2$.   We may assume that, in traversing $\alpha_j$ from one end to the other, we first traverse $\tau_1$ from its end in $P_{e_i}^1$ to its end $\times_{2,1}$ in $P_{e_i}^2$.  As we continue along $\alpha_j$ from $\times_{2,1}$, we are inside $\Delta_{\Side{e_i}2}$ until we meet the second crossing $\times_{2,2}$ of $\alpha_j$ and $P_{e_i}^2$.  

The earlier ``Useful Fact" asserts that $\times_{2,1}$ and $\times_{2,2}$ are the only crossings of $\alpha_j$ with $P_{e_i}^2$.  It follows that $\times_{2,2}$ is an end of $\tau_2$ and, continuing along $\alpha_{e_j}$ from $\times_{2,2}$ we are traversing $\tau_2$ up to its other end, which is in $P_{e_i}^1$.

In summary, $\alpha_{e_j}$ crosses $P_{e_i}^1$ at one end of $\tau_1$, crosses $P_{e_i}^2$ at the other end of $\tau_1$, then goes through $\Delta_{\Side{e_i}2}$ until it crosses $P_{e_i}^2$ a second (and final) time, beginning its traversal of $\tau_2$ up to the second (and final) crossing of $P_{e_i}^1$.  The rest of $\alpha_{e_j}$ is inside $\Delta_{\Side{e_i}1}$ and so its terminus must be in $C_{\Side{e_i}1}$.

However, Lemma \ref{lm:noneOrTwo} and \ref{it:edgeCrossings} imply $\alpha_j$ crosses $C_{\Side{e_i}1}$ only twice, and we have three crossings:  $\tau_1\cap P_{e_i}^1$, $\tau_2\cap P_{e_i}^1$, and the terminus of $\alpha_{e_j}$, a contradiction.
\end{proof}

The verification that $\alpha_{e_1},\dots,\alpha_{e_i}$ satisfy Conditions \ref{it:defAlpha} -- \ref{it:twoAlpha} is completed by showing that $\alpha_{e_i}$ does not cross any avoidable $\alpha_{e_j}$ more than once.  By way of contradiction, we assume that $\alpha_{e_i}$ crosses the avoidable $\alpha_{e_j}$ more than once.    Since $\alpha_{e_j}$ is avoidable, either it does not cross $P_{e_i}^1$ or it does not cross $P_{e_i}^2$; for the sake of definiteness, we assume the latter.  In particular, $\alpha_{e_j}$ does not cross $e_i$ and, therefore, must cross each of the subarcs $\alpha_{e_i}^{u}$ and $\alpha_{e_i}^{v}$ (see Figure \ref{fg:2crossingExtension}).

\begin{figure}[!ht]
\begin{center}
\includegraphics[scale=.3]{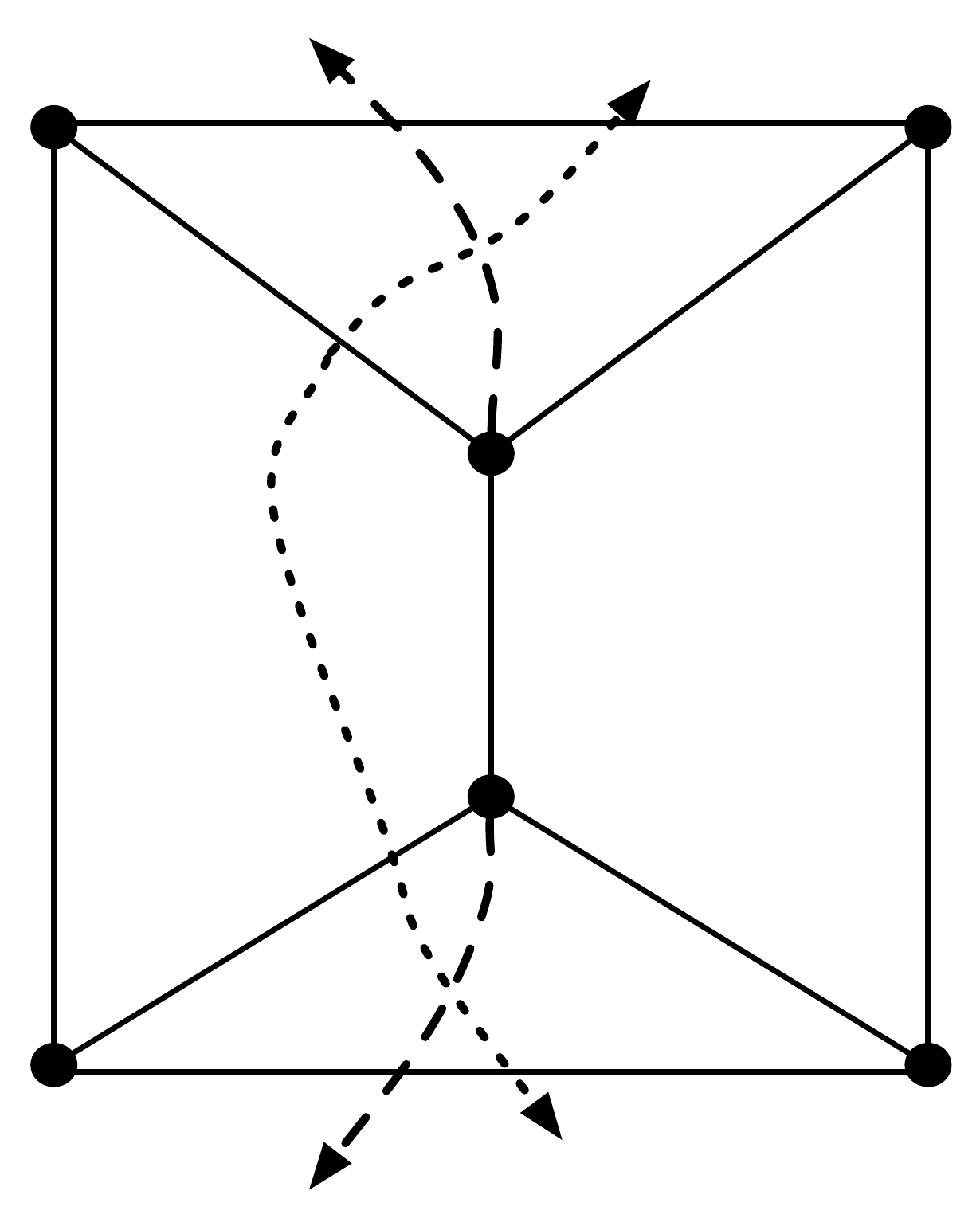}
\caption{One instance of $\alpha_{e_i}$ crossing an avoidable $\alpha_{e_j}$ twice.}\label{fg:2crossingExtension}
\end{center}
\end{figure}

By the choice of $\eta_{v}$, there is an unavoidable $\alpha_{e_{j'}}$ that crosses $f_{e_i}^{v}$ between the intersections of $\alpha_{e_j}$ and $\alpha_{e_i}$ on $f_{e_i}^{v}$.  Moreover, from its intersection with $f_{e_i}^{v}$, $\alpha_{e_{j'}}$ crosses $P_{e_i}^2\cap Q_{e_i}^{v}$ and, in going to that crossing, it must also cross the segment of $\alpha_{e_j}$ inside $Q_{e_i}^{v}$.  Thus, \ref{it:twoAlpha} implies that $\alpha_{e_{j'}}$ and $\alpha_{e_j}$ cannot cross again.  

As we follow $\alpha_{e_{j'}}$ from its end in $f_{e_i}^v$, we come first to the crossing with $\alpha_{e_j}$, and then to a crossing with $P_{e_i}^2$.  Continuing from this point, we cross $P_{e_i}^2$ again at $\times_{2,2}$ followed by the first crossing $\times_{1,1}$ with $P_{e_i}^1$.  Some point $\times_{e_i}$ of $\alpha_{e_{j'}}$ in the closed subarc between $\times_{2,2}$ and $\times_{1,1}$ is in $\alpha_{e_i}$.  Claim \ref{cl:unavoidableOnce} asserts that $\times_{e_i}$ is the unique crossing of $\alpha_{e_{j'}}$ with $\alpha_{e_i}$.

The point $\times_{e_i}$ must lie on the segment of $\alpha_{e_i}$ between the two crossings of $\alpha_{e_i}$ with $\alpha_{e_j}$, as otherwise $\alpha_{e_{j'}}$ must cross $\alpha_{e_j}$ a second time.  It follows that, as we continue a short distance along $\alpha_{e_{j'}}$ beyond $\times_{e_i}$, there is a point of $\alpha_{e_{j'}}$ that is inside the simple closed curve bounded by the segments of each of $\alpha_{e_i}$ and $\alpha_{e_j}$ between their two intersection points.  But $\alpha_{e_{j'}}$ must get to $C_F$ from here without crossing $\alpha_{e_i}\cup \alpha_{e_j}$, which is impossible, showing that $\alpha_{e_1},\dots,\alpha_{e_i}$ satisfy all of \ref{it:defAlpha}--\ref{it:twoAlpha}.

\ignore{*****

It follows that, except for the portion of $\alpha_{e_{j'}}$ from its end in $f_{e_i}^{v}$ to the crossing with $\alpha_{e_j}$, $\alpha_{e_{j'}}$ is contained in the simple closed curve containing $\alpha_{e_j}$ and the portion of $C_F$ joining those ends and containing $C_F\cap C_{\Side{e_i}2}$.

Because $\alpha_{e_{j'}}$ is unavoidable for $e_i$, it must cross $P_{e_i}^1$.  Traversing $\alpha_{e_{j'}}$ from its end in $f_{e_i}^{v}$, this must be after the crossing with $P_{e_i}^2\cap Q_{e_i}^{v}$.   Since $\alpha_{e_i}$ separates the interior of $\Delta_{\Side{e_i}2}$ from $\Delta_{\Side{e_i}1}$, this implies a crossing of $\alpha_{e_{j'}}$ with $\alpha_{e_i}$ no later than this first crossing of $P_{e_i}^1$.  The preceding paragraph shows that $\alpha_{e_{j'}}$ does not have its other end in $C_F\cap C_{\Side{uv}1}$ and, therefore, $\alpha_{e_{j'}}$ must cross $P_{e_1}^1$ a second time.

The fact that $\alpha_{e_{j'}}$ cannot cross $\alpha_{e_j}$ again implies that $\alpha_{e_{j'}}$ terminates in the portion of $C_F$ between the two ends of $\alpha_{e_j}$ that contains $C_F\cap C_{\Side{uv}2}$.   Again, $\alpha_{e_i}$ separates the second crossing of $\alpha_{e_{j'}}$ with $P_{e_i}^1$ from this portion of $C_F$, showing that $\alpha_{e_{j'}}$ crosses $\alpha_{e_i}$ a second time, contradicting Claim \ref{cl:unavoidableOnce}.
}

What remains is to deal with the portions of the pseudolines that are in $F$.    We begin by letting $\gamma$ be a circle so that $D[K_n]$ is contained in the interior of $\gamma$.  We label $C_F$ as $(v_0,f_1,v_1,\dots,f_k,v_0)$.  Our first step is to extend one at a time each $D[f_i]$ to an arc $\beta_{f_i}$ in $F\cup D[f_i]$ that, except for its endpoints, is contained in the open, bounded side of $\gamma$ joining antipodal points $a_i$ and $b_i$ on $\gamma$.  Pick arbitrarily two antipodal points $a_1$ and $b_1$  on $\gamma$ and extend $D[f_1]$ in $F$ to an arc $\beta_{f_1}$ joining $a_1$ and $b_1$.

Suppose we have $\beta_{f_1},\dots,\beta_{f_{i-1}}$.  The arc $\beta_{f_i}$ will have to cross $\beta_{f_{i-1}}$ at $v_{i-1}$.  (If $i=k$, then $\beta_{f_i}$ will also have to cross $\beta_{f_1}$ at $v_0$.)  Extend $D[f_i]$ slightly so that it actually crosses $\beta_{f_{i-1}}$ at $v_{i-1}$ (and $\beta_{f_1}$ at $v_0$ when $i=k$).  If $i<k$, then pick arbitrarily antipodal points  $a_{i}$ and $b_i$ on $\gamma$ distinct from $a_1,\dots,a_{i-1},b_1,\dots,b_{i-1}$ and join the endpoints of $f_i$ to these points, making sure to cross $\beta_{f_{i-1}}$ at $v_{i-1}$. 

If $i=k$, then $\beta_{f_1}$ and $\beta_{f_{k-1}}$ both constrain $\beta_{f_k}$.  In this case, $f_k$ is in precisely one of the four regions inside $\gamma$ created by $\beta_{f_1}$ and $\beta_{f_{k-1}}$.  The arc in $\gamma$ contained in the boundary of this region and its  antipodal mate are to be avoided.  The endpoints $a_k$ and $b_k$ of $\beta_{f_k}$ are in the other antipodal pair of arcs in $\gamma$.  We extend $D[f_k]$ slightly into each of these two regions.

In every case, we apply Lemma \ref{lm:dualPaths} to the two regions of $F\setminus \beta_{f_{i-1}}$.  In particular, both slight extensions of $D[f_i]$ are constrained not to cross $\beta_{f_{i-1}}$ except at $v_{i-1}$.  (For $f_k$, we restrict to the two of the four regions of $F\setminus (\beta_{f_{k-1}}\cup\beta_{f_1})$ that contain the slight extensions of $D[f_k]$.)

These restrictions guarantee that $\beta_{f_{i}}$ does not cross $\beta_{f_{i-1}}$ more than once.  Furthermore, $\beta_{f_i}$ does not cross any of $\beta_{f_1},\dots,\beta_{f_{i-2}}$ more than once in each of the subregions of $F$.  For $j=1,2,\dots,i-2$, $\beta_{f_i}$ and $\beta_{f_j}$ have interlaced ends in $\gamma$ and, therefore, they cross an odd number of times.  It follows that they cross at most once.

We use a similar process to extend the arcs $\alpha_{e_i}$ that join points in $C_F$.  Again, extend each one slightly into $F$ in such a way that, every time two $\alpha_{e_i}$'s meet at a common vertex in $C_F$, they cross at that vertex.   The slightly extended $\alpha_{e_i}$ crosses some of the $\beta_{f_j}$'s:  at least 2 (with equality if both endpoints of $\alpha_{e_i}$ are in the interiors of $f_j$'s) and at most 4 (with equality if both endpoints of $e_i$ are in $C_F$).  

We will be proceeding with the $\alpha_{e_i}$ one by one, so that, when extending $\alpha_{e_i}$, we already have extended $\alpha_{e_1},\dots,\alpha_{e_{i-1}}$, to arcs $\alpha^*_{e_1},\dots,\alpha^*_{e_{i-1}}$. 
Let $\Lambda$ be the set consisting of those $\beta_{f_j}$ and $\alpha^*_{e_k}$ that cross the slightly extended $\alpha_{e_i}$.  Since $\alpha_{e_i}$ crosses each arc in $\Lambda$ exactly once, the two ends of $\alpha_{e_i}$ are in two regions of $F\setminus  (\bigcup_{\lambda\in \Lambda}\lambda)$ that are incident with antipodal segments of $\gamma$.  Choose arbitrarily an antipodal pair, one from each of these antipodal segments.  Apply Lemma \ref{lm:dualPaths} to these two regions to extend the ends of $\alpha_{e_i}$ to these chosen antipodal points, yielding the arc $\alpha^*_{e_i}$.  

Evidently, $\alpha^*_{e_i}$ crosses every arc in $\Lambda$ exactly once.    None of the other $\beta_{f_j}$ and $\alpha^*_{e_k}$ crosses $\alpha_{e_i}$.  Therefore, $\alpha^*_{e_i}$ crosses each of these at most twice, once in each of the two regions in which $\alpha^*_{e_i}$ completes $\alpha_{e_i}$.  Since $\alpha^*_{e_i}$ and any $\beta_{f_j}$ or $\alpha^*_{e_k}$ have interlaced ends in $\gamma$, they cross an odd number of times; thus, they cross exactly once.
\end{cproofof}

\section{Empty triangles in face-convex drawings}\label{sec:emptyTriangles} 

Let $D$ be a drawing of $K_n$, let $xyz$ be a 3-cycle in $K_n$ and let $\Delta$ be an open disc bounded by $D[xyz]$.  Then $\Delta$ is an {\em empty triangle\/} if $D[V(K_n)]\cap \Delta=\varnothing$.     The classic theorem of B\'ar\'any and F\"uredi \cite{barany} asserts that, in any rectilinear drawing of $K_n$, there are $n^2+{}$O$(n\log n)$ empty triangles.

\change {In  Corollary \ref{co:BarFurFaceCon}, we extend the} B\'ar\'any and F\"uredi theorem to pseudolinear drawings by proving the same theorem as theirs for face-convex drawings.  Their proof adapts perfectly, as long as one has an appropriate ``intermediate value" property.  \change{ The other main result of this section is that} any convex (not necessarily face-convex) drawing of $K_n$ has at least $n^2/3+{}$O$(n)$ empty triangles.

Let $D$ be a convex drawing of $K_n$ and suppose that $T$ is a transitive orientation of $K_n$ with the additional property that, for each convex region $\Delta$ bounded by a triangle $(u,v,w,u)$, if $x$ is inside $\Delta$, then $x$ is neither a source nor a sink in the inherited orientation of the $K_4$ induced by $u,v,w,x$.   A convex drawing of $K_n$ with such a transitive orientation is {\em a convex intermediate value drawing\/}.

Convexity is not quite enough for our proof.  A convex drawing $D$ of $K_n$ is {\em hereditarily convex\/} if, for each triangle $T$ there is a specified side $\Delta_T$ of $D[T]$ that is convex and, moreover, for every triangle $T'\subseteq \Delta_T$, $\Delta_{T'}\subseteq \Delta_T$.  Every pseudolinear drawing is trivially hereditarily convex, but so also is the ``tin can" drawing of $K_n$ that has $H(n)$ crossings.  More generally, any drawing of $K_n$ in which arcs are drawn as geodesics in the sphere is hereditarily convex.

\begin{theorem}\label{th:generalBarFur}
A hereditarily convex intermediate value drawing of $K_n$ has $n^2+\textrm{\em O}(n\log n)$ empty triangles.
\end{theorem}

\begin{cproof}
Label $V(K_n)$ with $v_1,v_2,\dots,v_n$ to match the intermediate value orientation, so $\overrightarrow{v_iv_j}$ is the orientation precisely when $i<j$.   We henceforth ignore the arrow and use $v_iv_j$ for an edge only when $i<j$.

B\'ar\'any and F\"uredi \cite[Lemma 8.1]{barany} prove the following fact.

\begin{lemma}\label{lm:barany} Let $G$ be a graph with vertex set $\{1,2,\dots,n\}$.  Suppose that there are no four vertices $i<j<k<\ell$ such that $ik$, $i\ell$, and $j\ell$ are all edges of $G$.  Then $|E(G)|\le 3n\lceil \log n\rceil$. \eop
\end{lemma}

We will apply this lemma exactly as is done in \cite{barany}.   For distinct $i,j,k$ with $i<j$, say that $v_k$ is {\em to the left of $v_iv_j$\/} if the face of the triangle $v_iv_jv_k$ containing the left side (as we traverse $v_iv_j$) of $v_iv_j$ is convex.  

\begin{claim}\label{cl:existEmpty}
If $v_k$ is to the left of $v_iv_j$, then the face of the triangle $v_iv_jv_k$ containing the left side of $v_iv_j$ contains an empty triangle incident with $v_iv_j$.
\end{claim}

\begin{proof}
If $v_{\ell}$ is inside the convex triangle $\Delta$, then $v_{\ell}$ is joined within $\Delta$ to the 3 corners of $\Delta$ and the triangle incident with $v_\ell$ contained in $\Delta$, and incident with $v_iv_j$ is convex by heredity.  Since $\Delta$ contains a minimal convex triangle incident with $v_iv_j$, this one is necessarily empty.
\end{proof}

It follows that, as long as $v_iv_j$ has a vertex to the left and a vertex to the right, then $v_iv_j$ is in two empty triangles.   The main point is to show that there are not many pairs $(i,j)$ such that $i<j$ and $v_iv_j$ is in only one empty triangle $v_iv_jv_k$ with $i<k<j$.  We count those that have all intermediate vertices on the left.

\begin{claim}\label{cl:noFour} There do not exist $i<j<k<\ell$ such that: $v_j$ and $v_k$ are on the left of $v_iv_{\ell}$; $v_j$ is on the left of $v_iv_k$; $v_k$ is on the left of $v_jv_\ell$; and each of $v_iv_\ell$, $v_iv_k$, and $v_jv_\ell$ is in at most one empty triangle.
\end{claim}

\begin{proof}
Suppose by way of contradiction that:   $i<j<k<\ell$; $v_j$ and $v_k$ are on the left of $v_iv_{\ell}$; $v_j$ is on the left of $v_iv_k$; $v_k$ is on the left of $v_jv_\ell$; and each of $v_iv_\ell$, $v_iv_k$, and $v_jv_\ell$ is in at most one empty triangle.

Let $J$ be the complete subgraph induced by $v_i,v_j,v_k,v_\ell$.   Suppose first that $D[J]$ were a planar $K_4$.  In order for both $v_j$ and $v_k$ to be on the left of $v_iv_\ell$, one of $v_j$ and $v_k$ is inside the convex triangle $\Delta$ bounded by $v_i$, $v_\ell$, and the other of $v_j$ and $v_k$; let $v_{j'}$ be the one inside. Thus, the edges $v_iv_{j'}$ and $v_{j'}v_\ell$ are both inside $\Delta$.  (See Figure \ref{fg:planarK4}.)

\begin{figure}[!ht]
\begin{center}
\scalebox{1.0}{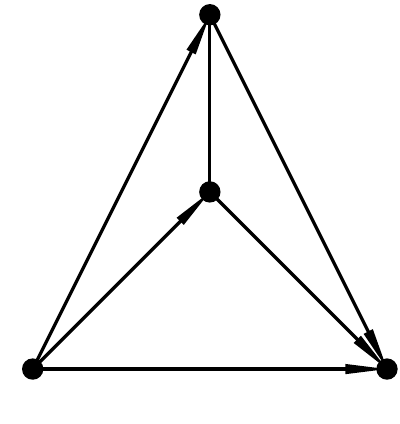}
\caption{The case $v_i,v_j,v_k,v_\ell$ make a planar $K_4$.}\label{fg:planarK4}
\end{center}
\end{figure}

\ignore{
\begin{figure}[!ht]
\begin{center}
\includegraphics[scale=.3]{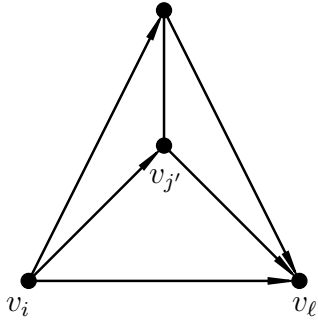}
\caption{The case $v_i,v_j,v_k,v_\ell$ make a planar $K_4$.}\label{fg:planarK4}
\end{center}
\end{figure}
}

The nested condition implies that the three smaller triangular regions  inside $\Delta$ and incident with $v_{j'}$ are all convex.   If $j'=j$, then $v_j$ is to the right of $vv_k$, a contradiction.  If $j'=k$, then $v_k$ is to the right of $v_jv_\ell$, a contradiction.  Therefore, $D[J]$ is not a planar $K_4$.

In a crossing $K_4$, the convex sides of any of the triangles in the $K_4$ are the unions of two regions incident with the crossing.  For each 3-cycle $T$ in this $K_4$, the fourth vertex shows that it is on the non-convex side of $T$, so it is the bounded regions in the figure that are convex.  With the assumption that $v_j$ and $v_k$ are both to the left of $v_iv_\ell$, we see that $v_iv_\ell$ is in the 4-cycle that bounds a face of $D[J]$.  Let $v_{j'}$ be the one of $v_j$ and $v_k$ that is the neighbour of $v_i$ in this 4-cycle.  (See Figure \ref{fg:nonPlanarK4}.)

\begin{figure}[!ht]
\begin{center}
\scalebox{1.0}{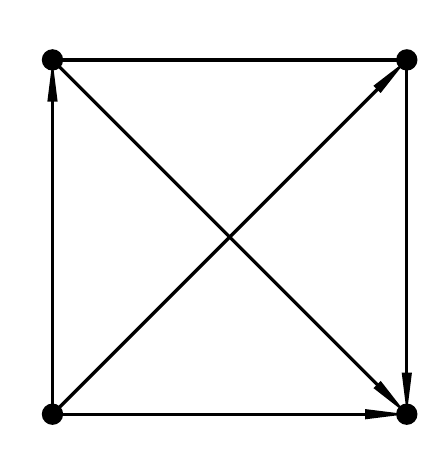}
\caption{The case $v,v_j,v_k,v_\ell$ make a nonplanar $K_4$.}\label{fg:nonPlanarK4}
\end{center}
\end{figure}

\ignore{
\begin{figure}[!ht]
\begin{center}
\includegraphics[scale=.3]{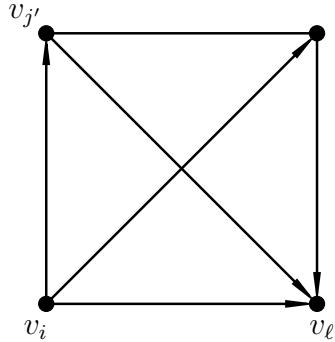}
\caption{The case $v,v_j,v_k,v_\ell$ make a nonplanar $K_4$.}\label{fg:nonPlanarK4}
\end{center}
\end{figure}
}

If $v_{j'}=v_k$, then $v_j$ is to the right of $v_iv_k$, a contradiction.  Therefore, $v_{j'}=v_j$.  Consider the quarter of the inside of the 4-cycle that is incident with $v_i$, $v_\ell$, and the crossing.  We claim that there is no vertex of $K_n$ in the interior of this region.

If $v_r$ were in this region, then $r<i$ or $r>\ell$ violates the intermediate value property.  If $i<r<k$, then heredity implies that we have the contradiction that $v_r$ is to the right of $v_iv_k$.  If $j<r<\ell$, then $v_r$ is to the right of $v_jv_\ell$, an analogous contradiction.  It follows that the convex triangles bounded by both $(v_i,v_\ell,v_j,v_i)$ and $(v_i,v_\ell,v_k,v_i)$ contain empty triangles and these empty triangles are different, the final contradiction.
\end{proof}

It follows from Claim \ref{cl:noFour} and Lemma \ref{lm:barany} that only O$(n\log n)$ pairs $i<j$ have the property that there is at most one $k$ such that $i<k<j$ and  $v_iv_kv_j$ is an empty triangle.  

Since there are $\binom n2$ pairs $i<j$, there are $\binom n2-\textrm{O}(n\log n)$ pairs with two empty triangles $v_iv_kv_j$ and $v_iv_{k'}v_j$, with $i<k,k'<j$.  Thus, there are $n^2-\textrm{O}(n\log n)$ empty triangles, as required.
\end{cproof}

The proof that the B\'ar\'any-F\"uredi result holds for pseudolinear drawings comes from showing that every face-convex drawing has a transitive ordering with the intermediate value property.

\begin{theorem}\label{th:faceConvTransOrd}  If $D$ is a face-convex drawing of $K_n$, then there is a transitive ordering with the intermediate value property.
\end{theorem}

\begin{cproof}
Let $v_1$ be any vertex incident with a face $F$ witnessing face-convexity.  Let $v_2,v_3,\dots,v_n$ be the cyclic rotation at $v_1$ induced by $D$, labelled so that $v_1v_2$ and $v_1v_n$ are incident with $F$.

We claim that this transitive ordering has the intermediate value property. Let $J$ be an isomorph of $K_4$ such that $D[J]$ is planar; let $i<j<k<\ell$ be such that the four vertices of $J$ are $v_i,v_j,v_k,v_\ell$.  Deleting $v_1,\dots,v_{i-1}$ and $v_{\ell+1},\dots,v_n$ shows that $v_i$ and $v_\ell$ are incident with the face of $D[J]$ that contains $F$.  The convex side of the triangle of $J$ bounding this face is, therefore, the side containing one of $v_j$ and $v_k$.  Evidently, this one is neither a sink nor a source of $J$.  
\end{cproof}

\change{Theorems \ref{th:generalBarFur} and \ref{th:faceConvTransOrd} immediately imply the generalization of B\'ar\'any and F\"uredi to face-convex drawings.

\begin{corollary}\label{co:BarFurFaceCon}  Let $D$ be a face-convex drawing of $K_n$.  Then $D$ has at least $n^2+\textrm{\em O}(n\log n)$ empty triangles.
\end{corollary}
}

We conclude this work by showing that convexity is enough to guarantee O$(n^2)$ empty triangles.  This is somewhat surprising, since it is known that general drawings of $K_n$ can have as few as $2n-4$ empty triangles \cite{fewEmpties}.

\begin{theorem}\label{th:convexEmptyTriangles}
Let $D$ be a convex drawing of $K_n$.  Then $D$ has $n^2/3 -{}${\em O}$(n)$ empty triangles.
\end{theorem}

\begin{cproof}
Let $U$ be the subgraph of $K_n$ induced by the edges not crossed in $D$.  Then $D[U]$ is a planar embedding of the simple graph $U$, so $U$ has at most $3n-6$ edges.  Thus, there are $n^2/2-{}$O$(n)$ edges that are crossed in $D$.

For each edge $e$ that is crossed in $D$, let $f$ be one of the edges that crosses $e$.  Let $J_{e,f}$ be the $K_4$ induced by the vertices of $K_n$ incident with $e$ and $f$.  Because $D[J_{e,f}]$ is a crossing $K_4$, each of the four triangles in $J_{e,f}$ has a unique convex side; it is the side not containing the fourth vertex.  Let $\Delta_{e,f}$ be the closed disc that is the union of these four convex triangles.

It follows that, for any other vertex $x$ that is on the convex side of any of these four triangles, $x$ is joined inside $\Delta_{e,f}$ to the four corners of $\Delta_{e,f}$, implying one of the edges $f'$ incident with $x$ crosses $e$.  Now $\Delta_{e,f'}$ is strictly contained in $\Delta_{e,f}$, proving that there is a minimal $\Delta_{e,f}$ that does not contain any vertex of $K_n$ in its interior.

The two convex triangles incident with $e$  and contained in such a minimal $\Delta_{e,f}$ are both empty. Thus, every crossed edge is in at least two empty triangles.  Since every empty triangle contains precisely three edges, there are at least 
\[
\frac 132\left(\frac {n^2}2-\textrm{O}(n)\right) = \frac{n^2}3 - \textrm{O}(n)
\]
empty triangles, as required.
\end{cproof}

\end{document}